\documentclass[reqno]{amsart}
\pdfoutput=1
\usepackage{graphicx}
\usepackage{amssymb, amsmath, amsthm}
\usepackage[bookmarks=false]{hyperref}
\usepackage{mathptmx, helvet, courier, graphicx, multicol, footmisc} 

\allowdisplaybreaks[1] 

\newtheorem{Theorem}[equation]{Theorem}{\bfseries}{\itshape}
\newtheorem*{rem}{Remark} 
\newenvironment{Remark}{\begin{rem}\normalfont}{\end{rem}}

\newcommand\sumk{{\left| \boldsymbol {k} \right|}}
\newcommand{\qrfac}[2]{{\left({#1}; q \right)_{#2}}} 
\newcommand{\pqrfac}[3]{{\left({#1};#3\right)_{#2}}}
\newcommand{\smallprod}[1]{\prod\limits_{r =1}^{#1}} 
\newcommand{\multisum}[2]{\underset{r=1,2, \dots, #2}{\sum\limits_{#1_r\geq 0}}} 

\begin{document} 
\title[A bibasic Heine transformation]{A bibasic Heine transformation formula and Ramanujan's $_2\phi_1$ transformations}
\author[Gaurav Bhatnagar]{Gaurav Bhatnagar*}
\thanks{Research supported in part by the Austrian Science Fund (FWF): F 5008-N15.}
\address{Fakult\"at f\"ur Mathematik,  Universit\"at Wien,
Oskar-Morgenstern-Platz 1, 1090 Wien, Austria.\\
}
\email{bhatnagarg@gmail.com}

\dedicatory{Dedicated to Krishnaswami Alladi on his 60th birthday}



\begin{abstract}
We study Andrews and Berndt's organization of Ramanujan's  transformation formulas in Chapter 1 of their book {\em Ramanujan's Lost Notebook, Part II}.  In the process, we rediscover a bibasic Heine's transformation, which follows from a Fundamental Lemma given by Andrews in 1966, and obtain identities proximal to Ramanujan's entries.  We also provide a multibasic generalization of Andrews' 1972 theorem concerning a $q$-analogue of the Lauricella function. Our results only require the $q$-binomial theorem, and are an application of what Andrews and Berndt call \lq Heine's Method'. 
\end{abstract}

%
\maketitle

\section{Introduction}
In Chapter 1, Part II of their edited version of Ramanujan's  \cite{Ramanujan-LN} Lost Notebook, Andrews and Berndt \cite{AB2009} have organized Ramanujan's transformation formulas related to Heine's  $_2\phi_1$ transformations. While studying their work, we discovered a large number of formulas that are proximal to Ramanujan's own entries. 

For example, one of Ramanujan's formulas is \cite[Entry 1.6.6]{AB2009}: for $|q|<1$,
\begin{align*}
\frac{1}{1-q}+\sum_{j=1}^{\infty} &
\frac {(-1)^j q^{j^2+j} }{\left( 1-q^{2j+1} \right) \left[(1-q^2)(1-q^4)(1-q^6)\cdots (1-q^{2j}) \right] } \cr
&= \sum_{k=0}^\infty q^{\frac{k(k+1)}{2}} \cr
&=1+q^1+q^3+q^6+q^{10}+q^{15}+\cdots.
\end{align*}
The right hand side is the well-known  theta function which Ramanujan denoted as $\psi(q)$. It has the product representation (see Berndt \cite[p.\ 11]{Berndt2006})
$$\psi(q):= \sum_{k=0}^\infty q^{\frac{k(k+1)}{2}}=
\prod_{k= 0}^\infty
\frac{(1-q^{2k+2})}{(1-q^{2k+1}) }.$$
We recover Ramanujan's Entry 1.6.6 and, in the same breath, obtain the formula
\begin{align*}
\frac{1}{1+q}+\sum_{j=i}^{\infty} &
\frac {(-1)^j q^{{j+1\choose 2}} }{ \left( 1+q^{j+1} \right) 
\left[(1-q)(1-q^2)(1-q^3)\cdots (1-q^{j}) \right] }\cr
&= \sum_{k=-\infty}^\infty (-1)^kq^{k^2} \cr
&= 1-2q+2q^4-2q^9+\cdots.
\end{align*}
Now the right hand side is  (in Ramanujan's notation) $\phi(-q)$, with product representation given by \cite[eq.\ (1.4.10)]{AB2009}
$$\phi(-q):= \sum_{k=-\infty}^\infty (-1)^kq^{k^2} =
\prod_{k= 1}^\infty
\frac{(1-q^{k})}{(1+q^{k}) }.$$

For our next example, we require some notation. 
The $q$-rising factorial is defined as $\qrfac{A}{0} :=1$, and when $k$ is a positive integer, 
\begin{equation*}
\qrfac{A}{k} := (1-A)(1-Aq)\cdots (1-Aq^{k-1}).
\end{equation*}
Notice that it is a product of $k$ terms. The parameter $q$ is called the \lq base'.  The infinite $q$-rising factorial is defined, for $|q|<1$, as
\begin{equation*}
\qrfac{A}{\infty} := \prod_{r=0}^{\infty} (1-Aq^r).
\end{equation*}
Observe that, for  $|q|<1$ \cite[eq.\ (I.5)]{GR90},
\begin{equation}\label{elementary1}
\qrfac{A}{k} = \frac{\qrfac{A}{\infty}}{\qrfac{Aq^k}{\infty}}.
\end{equation}
This is used to define $q$-rising factorials when $k$ is a complex number. 

With this notation, consider Ramanujan's formula \cite[Entry 1.4.17]{AB2009}
\begin{align*}\label{1.4.17}
\pqrfac{-aq}{\infty}{q} &
\sum_{j=0}^{\infty} 
\frac {b^j q^{{j+1\choose 2}} }{\pqrfac{q}{j}{q} \pqrfac{-aq}{tj}{q}} \cr
=&
\pqrfac{-bq}{\infty}{q} \sum_{k=0}^{\infty} 
\frac{ a^k q^{{k+1\choose 2}} }{\pqrfac{q}{k}{q} \pqrfac{-bq}{tk}{q}}
\end{align*}
and compare with the identity
\begin{align*}
\pqrfac{-aq^h}{\infty}{q^h} &
\sum_{j=0}^{\infty} 
\frac {b^j q^{t{j+1\choose 2}} }{\pqrfac{q^t}{j}{q^t} \pqrfac{-aq^h}{tj}{q^h}} \cr
=&
\pqrfac{-bq^t}{\infty}{q^t} \sum_{k=0}^{\infty} 
\frac{ a^k q^{h{k+1\choose 2}} }{\pqrfac{q^h}{k}{q^h} \pqrfac{-bq^t}{hk}{q^t}},
\end{align*}
obtained in our study. Here $|q|<1$ and $|q^t|<1$ in the first formula, and $|q^h|<1$, $|q^t|<1$ and 
$|q^{ht}|<1$ in the second. The reader may enjoy recovering Entry 1.4.17 from this formula. 

The objective of this paper is to report on our study of \cite[Ch.\ 1]{AB2009}. We are able to obtain 14 of Ramanujan's entries as immediate special cases of a particular transformation formula, and a large number of identities that are proximal to Ramanujan's own entries. In addition, we give a multibasic generalization of Andrews' 1972 formula for a $q$-Lauricella function and obtain a few interesting special cases, which again extend formulas of Ramanujan. 

During the course of our study, we stumbled upon the transformation formula
\begin{equation}\label{GB-Heine}
\sum_{k=0}^{\infty} \frac{\pqrfac{a}{k}{q^h}}{\pqrfac{q^h}{k}{q^h}}
\frac { \pqrfac{b}{hk}{q^t} }{\pqrfac{c}{hk}{q^t}} z^k 
=
\frac{\pqrfac{b}{\infty}{q^t}}{\pqrfac{c}{\infty}{q^t}}\frac{\pqrfac{az}{\infty}{q^h} }{\pqrfac{z}{\infty}{q^h}}
\sum_{j=0}^{\infty} \frac{\pqrfac{c/b}{j}{q^t}} {\pqrfac{q^t}{j}{q^t} }
\frac {\pqrfac{z}{tj}{q^h}}{\pqrfac{az}{tj}{q^h}} b^j,
\end{equation}
where $|z|<1$, $|b|<1$, and $h$ and $t$ are complex numbers such that $|q^h|<1$, $|q^t|<1$ and $|q^{ht}|<1$. 
Andrews and Berndt \cite{AB2009} use the $t=1$ case of this result (a formula due to Andrews \cite[Lemma 1]{An1966b}) often combined with the $h=1$ and $t=1$ case (a famous transformation of Heine, see Gasper and Rahman  \cite[Eq.~1.4.1]{GR90}). But these authors seem to have missed writing down \eqref{GB-Heine} explicitly, even though it can be proved in the same manner as Heine's result, and indeed follows from a very general approach to Heine's ideas, which Andrews \cite{An1966a} calls his \lq Fundamental Lemma'. This useful and simple identity may be a special case of a 50 year old identity, but it has not shown up in the standard textbook by Gasper and Rahman \cite{GR90}, and perhaps deserves to be highlighted. And so, in \S\ref{sec:heine-method}, we attempt a brief introduction. 

The plan for the rest of the paper is as follows. In \S \ref{sec:ramanujan2p1} and \S \ref{sec:1.6.6}, we report on our study of Ramanujan's transformation formulas. This part of our work can be considered to be an addendum to Chapter 1 of Andrews and Berndt \cite{AB2009}. 
 In \S \ref{sec:lauricella}, we closely follow ideas from Andrews \cite{An1972} to extend our work to multiple series that extend $q$-analogs of the Lauricella functions. We give a multibasic generalization of Andrews' formula \cite[eq.~(4.1)]{An1972}, and give several generalizations of two of Ramanujan's identities.

Before proceeding to Ramanujan's $_2\phi_1$ transformations, we consider \eqref{GB-Heine} again from the perspective of Heine's original ideas, an approach that Andrews and Berndt \cite{AB2009} have dubbed  \lq Heine's method'.   

\section{Heine's method: Transformations of Heine, Ramanujan and Andrews}\label{sec:heine-method}
This section is an introduction to Identity \eqref{GB-Heine}. We begin with a famous transformation formula of Heine that he found in 1847.
Heine's transformation formula \cite[eq.~78]{EH1847}  is
\begin{equation}\label{Heine-original}
\frac{\pqrfac{cx}{\infty}{q}}{\pqrfac{bx}{\infty}{q}}
\sum_{k=0}^{\infty} \frac{\pqrfac{a}{k}{q} \pqrfac{bx}{k}{q} }
{\pqrfac{q}{k}{q}\pqrfac{cx}{k}{q}}z^k 
=
\frac{\pqrfac{az}{\infty}{q} }{\pqrfac{z}{\infty}{q}}
\sum_{j=0}^{\infty} \frac{\pqrfac{c/b}{j}{q}\pqrfac{z}{j}{q}} {\pqrfac{q}{j}{q}\pqrfac{az}{j}{q} }
 (bx)^j.
\end{equation}
This is almost as Heine himself wrote it, except that he wrote $q^\alpha$, $q^\beta$ and $q^\gamma$ in place of $a, b$ and $c$.  Usually, this formula is stated with $x=1$, see Gasper and Rahman 
\cite[eq.\ (1.4.1)]{GR90}. 

Heine's formula was rediscovered by Ramanujan. It appears as Entry 6 in Chapter 16 of his second notebook, see Berndt \cite[p.~15]{Berndt1991-RN3}. In addition, there is  another transformation formula of Ramanujan resembling \eqref{Heine-original}. It appeared on Page 3 of the famous Lost Notebook \cite{Ramanujan-LN} (see \cite[Entry 1.4.1]{AB2009}),
and is dated circa 1919, going by Andrews and Berndt's \cite[p.~4]{AB2005} remarks on the likely timing of work presented in the Lost Notebook.
\begin{align} \label{1.4.1}
\frac{\qrfac{aq}{\infty}\pqrfac{cq}{\infty}{q^2}}{\qrfac{-bq}{\infty}\pqrfac{dq^2}{\infty}{q^2}} 
&
\sum_{j=0}^\infty \frac{\qrfac{-bq/a}{j}}{\qrfac{q}{j}}\frac{\pqrfac{dq^2}{j}{q^2}}{\pqrfac{cq}{j+1}{q^2}} (aq)^j \notag \\
&=\sum_{k=0}^\infty \frac{\pqrfac{cq/d}{k}{q^2}}{\pqrfac{q^2}{k}{q^2}}
\frac{\pqrfac{aq}{2k}{q}}{\pqrfac{-bq}{2k+1}{q}} (dq^2)^k.
\end{align}
This is Ramanujan's Entry 1.4.1 and it resembles Heine's transformation (coincidentally,  
eq.\ (1.4.1) of \cite{GR90}). Both the series have two products each in the numerator and denominator, and there are four infinite products outside the sums. However, some of the factors in the sums have base $q^2$ rather than 
$q$, and the number of terms in some of the factors of the summands are different. For example, notice the product $(aq;q)_{2k}$, a product of $2k$ factors in the summand on the right hand side of \eqref{1.4.1}. 

Andrews and Berndt \cite{AB2009} study many of Ramanujan's transformation formulas (in particular \eqref{1.4.1})  in Chapter 1, Part II of their series of books on Ramanujan's Lost Notebook. A key component of their study is Andrews' 1966 transformation formula \cite[Lemma 1]{An1966b}:
\begin{equation}\label{andrews1}
\sum_{k=0}^{\infty} \frac{\pqrfac{a}{k}{q^h}}{\pqrfac{q^h}{k}{q^h}}\frac { \qrfac{b}{hk} }{\qrfac{c}{hk}} z^k 
=
\frac{\qrfac{b}{\infty}}{\qrfac{c}{\infty}}\frac{\pqrfac{az}{\infty}{q^h} }{\pqrfac{z}{\infty}{q^h}}
\sum_{j=0}^{\infty} \frac{\qrfac{c/b}{j}} {\qrfac{q}{j} }\frac {\pqrfac{z}{j}{q^h}}{\pqrfac{az}{j}{q^h}} b^j,
\end{equation}
where $h=1, 2, 3, \dots$. Andrews' formula contains both \eqref{Heine-original} and  \eqref{1.4.1}. This can be seen by taking $h=1$ and $h=2$, respectively.  Andrews' transformation can also be found in \cite[Th.~1.2.1, p.~6]{AB2009} and  \cite[Ex. 3.35, p.~111]{GR90}.

Now, inspired by Heine's formulation \eqref{Heine-original}, we write Andrews' transformation more symmetrically as follows.
\begin{equation*}
\frac{\qrfac{bw}{\infty}}{\qrfac{w}{\infty}}
\sum_{k=0}^{\infty} \frac{\pqrfac{a}{k}{q^h}}{\pqrfac{q^h}{k}{q^h}}\frac { \qrfac{w}{hk} }{\qrfac{bw}{hk}} z^k 
=
\frac{\pqrfac{az}{\infty}{q^h} }{\pqrfac{z}{\infty}{q^h}}
\sum_{j=0}^{\infty} \frac{\qrfac{b}{j}} {\qrfac{q}{j} }\frac {\pqrfac{z}{j}{q^h}}{\pqrfac{az}{j}{q^h}} w^j.
\end{equation*}
This form suggests a further generalization of \eqref{andrews1}, where now we have terms involving two bases $q^h$ and $q^t$ (and hence the adjective {\em bibasic}).
\begin{Theorem}[A bibasic Heine transformation]
\label{th:GB-sym-Heine} Let $q$, $a$, $b$, $h$ and $t$ be complex numbers such that $|q^h|<1$, $|q^t|<1$ and $|q^{ht}|<1$, and suppose that the denominators in \eqref{GB-sym-Heine} are not zero. Then for $|w|<1$ and $|z|<1$, 
\begin{equation}\label{GB-sym-Heine}
\frac{\pqrfac{bw}{\infty}{q^t}}{\pqrfac{w}{\infty}{q^t}}
\sum_{k=0}^{\infty} \frac{\pqrfac{a}{k}{q^h}}{\pqrfac{q^h}{k}{q^h}}\frac { \pqrfac{w}{hk}{q^t} }{\pqrfac{bw}{hk}{q^t}} z^k 
=
\frac{\pqrfac{az}{\infty}{q^h} }{\pqrfac{z}{\infty}{q^h}}
\sum_{j=0}^{\infty} \frac{\pqrfac{b}{j}{q^t}} {\pqrfac{q^t}{j}{q^t} }\frac {\pqrfac{z}{tj}{q^h}}{\pqrfac{az}{tj}{q^h}} w^j.
\end{equation}
\end{Theorem}
\begin{Remark}
Replace $w$ by $b$ and $b$ by $c/b$ in \eqref{GB-sym-Heine} to obtain the form \eqref{GB-Heine} of the identity.
\end{Remark}

Before heading into the proof of Theorem \ref{th:GB-sym-Heine}, we make a few comments on the convergence 
of the series and products appearing in this identity. 

Observe that we require the conditions
$|q^t|<1$ and $|q^h|<1$ for the convergence of the infinite products $\pqrfac{w}{\infty}{q^t}$ and $\pqrfac{z}{\infty}{q^h}$. In view of \eqref{elementary1}, we require these conditions for the definition of products such as  $\pqrfac{w}{hk}{q^t}$ too. 

Next, note that the function $f(w):=\pqrfac{w}{\infty}{q}$ is a continuous function of $w$ in a neighborhood of $w=0$, and $f(0)=1.$ This follows from the fact that for fixed $q$, with $0<|q|<1$, the sequence of partial products 
$$f_k(w)=\prod\limits_{r=0}^{k-1}\left(1-wq^{r}\right)$$
converges absolutely to $f(w)$, and the convergence is uniform in a closed disk around $w=0$ contained in the unit disk $\{w \in {\mathbb C} : |w|<1\}$. 

Now we consider a factor such as $\pqrfac{w}{hk}{q^t}$, and show that if $|q^{ht}|<1$, then for large enough $k$, $|\pqrfac{w}{hk}{q^t}|$ is approximately equal to
$|\pqrfac{w}{\infty}{q^t}|$. 

By definition, we have
$$\pqrfac{w}{hk}{q^t} =\frac{\pqrfac{w}{\infty}{q^t}}{\pqrfac{wq^{thk}}{\infty}{q^t}}.$$
Now since $|q^{ht}|<1$, we must have $|q^{htk}|\to 0$ as $k\to \infty$, and thus, by the continuity of 
$f(w)$, $\pqrfac{wq^{thk}}{\infty}{q^t}\to \pqrfac{0}{\infty}{q^t}=1$. Thus for large enough $k$, 
$|\pqrfac{w}{hk}{q^t}|$ is approximately 
$|\pqrfac{w}{\infty}{q^t}|$. 

Using the above remarks, we can consider the absolute convergence of the series appearing on either side of \eqref{GB-sym-Heine}. Consider first the left hand side of \eqref{GB-sym-Heine}. We replace all the $q$-rising factorials in the summand by ratios of infinite products, using \eqref{elementary1}. Then we find that for large enough $k$, the absolute value of the  summand is bounded by a constant times the factor
$ |z|^k.$
Since the geometric series 
$$\sum_{k=0}^\infty z^k$$
converges absolutely for $|z|<1$, the sum on the left hand side of \eqref{GB-sym-Heine} converges absolutely for $|z|<1$. Similarly, the sum on the right hand side converges absolutely for $|w|<1$. 

To summarize, we have the conditions $|q^h|<1$, $|q^t|<1$, $|q^{ht}|<1$, $|z|<1$ and $|w|<1$ for the convergence of the products and series. 

We now proceed with the proof of the theorem. 
Theorem \ref{th:GB-sym-Heine} can be obtained as a very special case of Andrews' \cite{An1966a}  Fundamental Lemma (see our remark below). But we prove it on the lines of the proof of Heine's own proof of his transformation formula, which Andrews and Berndt \cite{AB2009} call Heine's method. We only require the identity \eqref{elementary1}
and 
 the $q$-binomial theorem \cite[eq. (1.3.2)]{GR90}: For $|z|<1$, $|q|<1$
\begin{equation}\label{q-bin}
\frac{\qrfac{az}{\infty}}{\qrfac{z}{\infty}} = 
\sum_{k=0}^{\infty} 
\frac{\qrfac{a}{k}}{\qrfac{q}{k}} z^k.
\end{equation}
\begin{proof}[Proof of Theorem \ref{th:GB-sym-Heine}] 
We begin with the left hand side of \eqref{GB-sym-Heine}.
\begin{align*}
\frac{\pqrfac{bw}{\infty}{q^t} }{\pqrfac{w}{\infty}{q^t}} 
\sum_{k=0}^{\infty} & \frac{\pqrfac{a}{k}{q^h}}{\pqrfac{q^h}{k}{q^h}}
\frac { \pqrfac{w}{hk}{q^t} }
{\pqrfac{bw}{hk}{q^t}} z^k  \cr
& = \sum_{k=0}^{\infty}  \frac{\pqrfac{a}{k}{q^h} }{\pqrfac{q^h}{k}{q^h}}z^k 
\frac{\pqrfac{bwq^{htk}}{\infty}{q^t} }{\pqrfac{wq^{htk}}{\infty}{q^t}} 
\quad \text{ (using \eqref{elementary1})}
 \cr
&= 
\sum_{k=0}^{\infty} \frac{\pqrfac{a}{k}{q^h} }{\pqrfac{q^h}{k}{q^h}}z^k 
\sum_{j=0}^{\infty} \frac{\pqrfac{b}{j}{q^t}  }{\pqrfac{q^t}{j}{q^t}} (wq^{htk})^j
\quad \text{ (using \eqref{q-bin})}
\cr
 &= 
\sum_{j=0}^{\infty} \frac{\pqrfac{b}{j} {q^t} }{\pqrfac{q^t}{j}{q^t}} w^j
\sum_{k=0}^{\infty} \frac{\pqrfac{a}{k}{q^h} }{\pqrfac{q^h}{k}{q^h}}(zq^{htj})^k 
 \cr
 &= 
 \sum_{j=0}^{\infty} \frac{\pqrfac{b}{j}{q^t}  }{\pqrfac{q^t}{j}{q^t}} w^j
\frac{\pqrfac{azq^{htj}}{\infty}{q^h} }{\pqrfac{zq^{htj}}{\infty}{q^h}}
\quad \text{ (using \eqref{q-bin} again)}
 \cr
 &=\frac{\pqrfac{az}{\infty}{q^h} }{\pqrfac{z}{\infty}{q^h}}
\sum_{j=0}^{\infty} \frac{\pqrfac{b}{j}{q^t}} {\pqrfac{q^t}{j}{q^t} }
\frac {\pqrfac{z}{tj}{q^h}}{\pqrfac{az}{tj}{q^h}} w^j.
    \end{align*}
Observe that $|w|<1$ and $|q^{ht}|<1$ implies $|wq^{htk}|<1$. Similarly, we must have
$|zq^{htj}|<1$. These conditions are required for the absolute convergence of the $q$-binomial series used here, and to justify the interchange of summation. 
\qed
\end{proof}

\begin{Remark} Andrews \cite[Lemma 1]{An1966b} (see also \cite[Ex. 3.35]{GR90}) mentions that the formula \eqref{andrews1} is valid when $h=1, 2, 3, \dots$. However, as we have seen, with sufficient conditions, we can take $h$ to be a complex number in \eqref{andrews1}. 
\end{Remark}

Observe that the $b=c$ case of Heine's transformation is \eqref{q-bin}, the  $q$-binomial theorem. When $b=c$, the summand contains the factor $\qrfac{1}j$ that is $1$ when $j=0$, and $0$ when $j>0$. Thus the sum on the right hand side of \eqref{Heine-original} reduces to $1$, and we obtain \eqref{q-bin}. 

A key property of Heine's transformation is that it can be iterated, and the process of iteration leads to symmetries of the sum which are useful in many contexts. See Gasper and Rahman \cite[eqs.\ (1.4.2) and (1.4.5)]{GR90}. 
Unfortunately, equation \eqref{GB-sym-Heine} cannot be iterated, making it less useful than Heine's transformation.   However, there is a bibasic version of a special case of Heine's second iterate due to Guo and Zeng \cite[Th. 2.2]{GZ2015}. 

There are also bibasic transformation formulas due to Gasper \cite[eq.\ (1.12)]{GG1989} (reproduced in \cite[Ex.\ 3.20]{GR90}). These  consist of four sums that are equal to each other. By equating the second and fourth sum, we get a formula equivalent to Andrews' transformation formula. Replace $p$ by $q^h$ in Gasper's transformation to obtain an equivalent form of \eqref{andrews1}.

\begin{Remark} 
Andrews  stated and used \eqref{andrews1} in \cite{An1966b}, and derived it using Theorem A of \cite{An1966a}, which in turn is derived from his \lq Fundamental Lemma'.  This lemma is really a most general approach to Heine's method, and should be better known. Andrews' \cite{An1966a} Fundamental Lemma can be stated as:
\begin{align}\label{Andrews-FL}
\sum_{k=0}^{\infty}  \frac{\pqrfac{a}{rk+s}{q}}{\pqrfac{q}{rk+s}{q}}
\frac { \pqrfac{b}{uk+v}{p} }{\pqrfac{c}{uk+v}{p}} z^k 
=
\frac{1}{r} & \frac{\pqrfac{b}{\infty}{p}}{\pqrfac{c}{\infty}{p}} 
\sum_{t=0}^{r-1} \omega_r^{-st} z^{-s/r} \cr
\times &
 \sum_{j=0}^{\infty} \frac{\pqrfac{c/b}{j}{p}} {\pqrfac{p}{j}{p} } 
\frac {\pqrfac{a\omega_r^t z^{1/r}p^{uj/r}}{\infty}{q}}{\pqrfac{\omega_r^t z^{1/r}p^{uj/r}}{\infty}{q}}
\left(bp^{v-us/r}\right)^j, \cr
\end{align}
where $\omega_r = e^{2\pi i/r}$ or some other primitive $r$th  root of unity, and  we assume the parameters satisfy suitable conditions to guarantee convergence of the two series. 

Equation \eqref{GB-Heine} can be obtained as a special case of \eqref{Andrews-FL}. Take $r=1$, $u=h$, $s=0=v$, $q\mapsto q^h$, and $p\mapsto q^t$ to obtain the second last step
(suitably re-labeled) in our proof of Theorem~\ref{th:GB-sym-Heine}. Professor Krattenthaler has remarked that, in fact, \eqref{Andrews-FL} follows from \eqref{GB-Heine} by \lq sectioning\rq\ the series on the left (a process described in our remark in \S\ref{sec:1.6.6}). In other words, Andrews Fundamental Lemma is equivalent to \eqref{GB-Heine}.  This involves recognizing that  we can write the factors in the sums with two independent bases $q$ and $p$, since $h$ and $t$ are complex numbers. Indeed, with these considerations we can re-write \eqref{GB-Heine}  as the $r=1$ case of \eqref{Andrews-FL}.  

The reader may enjoy proving \eqref{Andrews-FL} directly using Heine's method and sectioning.  
\end{Remark}

This completes our introduction to \eqref{GB-Heine}. We now consider special cases related to Ramanujan's transformations. In the rest of the paper, when stating special cases of \eqref{GB-sym-Heine}, we do not always explicitly state all the applicable convergence conditions mentioned in Theorem \ref{th:GB-sym-Heine}.

\section{Special cases inspired by Ramanujan's $_2\phi_1$ transformations}\label{sec:ramanujan2p1}

While studying Andrews and Berndt \cite[ch.~1]{AB2009}, we realized that many of Ramanujan's transformations in \cite[\S 1.4]{AB2009} are immediate special cases of  Ramanujan's transformation \eqref{1.4.1}, where one takes limits or special cases such as $a\to 0$, $b=0$, $c=0$ and $d\to 0$ and combinations of these. So we first re-write the bibasic Heine transformation in the form of Ramanujan's Entry 1.4.1, with a view to study its special cases. We will find that 
several of Ramanujan's entries in Chapter 1 of  \cite{AB2009} are immediate  special cases.  In addition, we note new identities that resemble Ramanujan's formulas.
\subsection*{Entry 1.4.1}
First we write \eqref{GB-Heine} in the form of Ramanujan's formula, by taking  $a\mapsto cq/d$, $b\mapsto aq^t$, $c\mapsto -bq^{t+1}$, and $z\mapsto dq^h$. Now divide both sides by $1+bq$ and multiply and divide the RHS by $1-cq$ and interchange the sides to
obtain
\begin{align}\label{GB-1.4.1}
\frac{\pqrfac{aq^t}{\infty}{q^t} \pqrfac{cq}{\infty}{q^h} }{\pqrfac{-bq}{\infty}{q^t}\pqrfac{dq^h}{\infty}{q^h}} &
\sum_{j=0}^{\infty} \frac{\pqrfac{-bq/a}{j}{q^t}} {\pqrfac{q^t}{j}{q^t} }
\frac {\pqrfac{dq^h}{tj}{q^h}}{\pqrfac{cq}{tj+1}{q^h}} (aq^t)^j \cr
=&
\sum_{k=0}^{\infty} \frac{\pqrfac{cq/d}{k}{q^h}}{\pqrfac{q^h}{k}{q^h}}
\frac{ \pqrfac{aq^t}{hk}{q^t} }{\pqrfac{-bq}{hk+1}{q^t}} (dq^h)^k .
\end{align}
Again, $h$ and $t$ are complex numbers, and we have the conditions $|q^h|<1$, $|q^t|<1$ and $|q^{ht}|<1$.  Further, for the series to converge, we require $|aq^t|<1$ and $|dq^h|<1$.  

Note that when $h=2$ and $t=1$, this reduces to \eqref{1.4.1}, Ramanujan's Entry 1.4.1.
The rest of Ramanujan's entries presented below are also special cases of \eqref{GB-1.4.1}.

\subsection*{Entry 1.4.2}
In equation \eqref{GB-1.4.1} take $a=d=1$, replace $c$ by $a$, and bring the product 
$\pqrfac{-bq}{\infty}{q^t}$ to the other side.  In this manner, we obtain a generalization of Entry 1.4.2:
\begin{align}\label{GB-1.4.2}
\frac{\pqrfac{q^t}{\infty}{q^t} \pqrfac{aq}{\infty}{q^h} }{\pqrfac{q^h}{\infty}{q^h}} &
\sum_{j=0}^{\infty} \frac{\pqrfac{-bq}{j}{q^t}} {\pqrfac{q^t}{j}{q^t} }
\frac {\pqrfac{q^h}{tj}{q^h}}{\pqrfac{aq}{tj+1}{q^h}} q^{tj} \cr
=&
\pqrfac{-bq}{\infty}{q^t} \sum_{k=0}^{\infty} \frac{\pqrfac{aq}{k}{q^h}}{\pqrfac{q^h}{k}{q^h}}
\frac{ \pqrfac{q^t}{hk}{q^t} }{\pqrfac{-bq}{hk+1}{q^t}} q^{hk} .
\end{align}

To take special cases, we use the following elementary identities from Gasper and Rahman~\cite[eq.~(I.27)]{GR90}:
\begin{equation}\label{GRI.27}
\qrfac{a}{rk} = \pqrfac{a, aq, aq^2, \dots, aq^{r-1}}{k}{q^r}, 
\end{equation}
and \cite[eq.~(I.30)]{GR90}
\begin{equation}\label{GRI.30}
\pqrfac{a^r}{k}{q^r} = \qrfac{a, a\omega_r, a\omega_r^2, \dots, a\omega_r^{r-1}}{k},
\end{equation}
where $\omega_r = e^{2\pi i /r}$ or some other primitive $r$th root of unity; here, we use the shorthand notation
$$\qrfac{a_1, a_2, \dots, a_n}{k} =\qrfac{a_1}{k}\qrfac{a_2}{k}\cdots \qrfac{a_n}{k}.$$

When $t=1$, and $h$ is a natural number bigger than $1$, then \eqref{GB-1.4.2} reduces to
\begin{align} \label{GB-1.4.2-h}
\left(q,q^2, \dots,  q^{h-1}; q^h\right)_{\infty} &  \pqrfac{aq}{\infty}{q^h} 
\sum_{j=0}^\infty \frac{\qrfac{-bq}{j}}{\pqrfac{aq}{j+1}{q^h}} (q\omega_h ,q\omega_h^2, \dots, q\omega_h^{h-1}; q)_{j}  q^j \notag \\
&= \qrfac{-bq}{\infty} \sum_{k=0}^\infty \frac{\pqrfac{aq}{k}{q^h}}{\pqrfac{-bq}{hk+1}{q}} 
(q,q^2, \dots, q^{h-1}; q^h)_{k} q^{hk},
\end{align}
where $\omega_h = e^{2\pi i /h}$ or some other primitive $h$th root of unity.  We have used
 \eqref{GRI.27} and \eqref{GRI.30} to write this expression.

When $h=2$, the primitive $h$th root of unity reduces to $-1$, and we obtain Ramanujan's transformation \cite[Entry 1.4.2]{AB2009}:
\begin{align} \label{1.4.2}
\left(q ; q^2\right)_{\infty}  \pqrfac{aq}{\infty}{q^2} 
\sum_{j=0}^\infty & \frac{\qrfac{-bq}{j}(-q ; q)_{j} }{\pqrfac{aq}{j+1}{q^2}}  q^j \cr
&= \qrfac{-bq}{\infty} \sum_{k= 0}^\infty \frac{\pqrfac{aq}{k}{q^2}(q; q^2)_{k}}{\pqrfac{-bq}{2k+1}{q}} 
 q^{2k}.
\end{align}
Observe that the denominator of the sum on the left does not contain the (usually) mandatory term $\pqrfac{q}{j}{q}$. This term is required to  terminate the series naturally from below, because
$$\frac{1}{\pqrfac{q}{j}{q}} = 0 \text{ whenever } j<0.$$
The same is true for the right hand side. This seems to be the motive for considering this special case. See also \eqref{1.4.10}, \eqref{1.4.11}, and Entry 1.6.5 (and related identities) below. 

\subsection*{Entry 1.4.5}
If we set $a=0$ in \eqref{GB-1.4.2}, replace $b$ by $a$, and bring all the infinite products to the right, we obtain a generalization of  Entry 1.4.5:
\begin{align}\label{GB-1.4.5}
\sum_{j=0}^{\infty}  & \frac{\pqrfac{-aq}{j}{q^t} \pqrfac{q^h}{tj}{q^h}} {\pqrfac{q^t}{j}{q^t} }
q^{tj} \cr
=&
\frac{\pqrfac{-aq}{\infty}{q^t}\pqrfac{q^h}{\infty}{q^h}  }{\pqrfac{q^t}{\infty}{q^t}} 
\sum_{k=0}^{\infty}
\frac{ \pqrfac{q^t}{hk}{q^t} }{\pqrfac{q^h}{k}{q^h}\pqrfac{-aq}{hk+1}{q^t}} q^{hk} .
\end{align}
When $t=1$,  and $h>1$ is a positive integer, this reduces to
\begin{align} \label{GB-1.4.5-t}
   \sum_{j=0}^\infty \qrfac{-aq}{j}  & (q\omega_h ,q\omega_h^2, \dots, q\omega_h^{h-1}; q)_{j}  q^j
= (q\omega_h ,q\omega_h^2, \dots, q\omega_h^{h-1}; q)_{\infty}  \qrfac{-aq}{\infty}  \notag \\
 &\times \sum_{k=0}^\infty \frac{1}{\pqrfac{-aq}{hk+1}{q}} 
(q,q^2, \dots, q^{h-1}; q^h)_{k} q^{hk},
\end{align}
where $\omega_h = e^{2\pi i /h}$ or some other primitive $h$th root of unity.  
Further take $h=2$ to obtain Ramanujan's formula \cite[Entry 1.4.5]{AB2009}:
\begin{equation} \label{1.4.5}
   \sum_{j=0}^\infty \qrfac{-aq}{j}  (-q; q)_{j}  q^j
= \left(-q; q\right)_{\infty}  \qrfac{-aq}{\infty} 
  \sum_{k=0}^\infty \frac{\left(q ; q^2\right)_{k} }{\pqrfac{-aq}{2k+1}{q}} 
q^{2k}.
\end{equation}
Instead, when $t=2$ and $h=1$ in \eqref{GB-1.4.5}, we obtain the formula
\begin{equation} \label{GB-entry1.4.5}
   \sum_{j=0}^\infty \pqrfac{-aq}{j}{q^2}  \left(q; q^2\right)_{j}  q^{2j}
= \left(q; q^2\right)_{\infty}  \pqrfac{-aq}{\infty}{q^2} 
  \sum_{k= 0}^\infty \frac{\left(-q ; q\right)_{k} }{\pqrfac{-aq}{k+1}{q^2}} 
q^{k}.
\end{equation}

\subsection*{Entry 1.4.3 and Entry 1.4.4}
There is a common generalization of Entry 1.4.3 and Entry 1.4.4. Take $b=0$, $d\to 0$ in \eqref{GB-1.4.1}, replace $c$ by $b/q^h$ and cancel $1-b/q^{h-1}$ on the LHS. Bring the products to the RHS to obtain, for $|aq^t|<1$,
\begin{equation}\label{GB-1.4.3-4}
 \sum_{j=0}^{\infty} 
 \frac{\left( aq^t\right)^j }{\pqrfac{q^t}{j}{q^t} \pqrfac{bq}{tj}{q^h}}
=
\frac{1}{\pqrfac{aq^t}{\infty}{q^t}\pqrfac{bq}{\infty}{q^h}}
 \sum_{k=0}^{\infty}  \frac{\pqrfac{aq^t}{hk}{q^t} }{\pqrfac{q^h}{k}{q^h}}
 (-bq)^k  q^{h{k \choose 2}} .
\end{equation}
When $h=2$ and $t=1$, \eqref{GB-1.4.3-4} reduces to \cite[Entry 1.4.3]{AB2009}:
\begin{equation} \label{1.4.3}
\sum_{j =0}^\infty \frac{a^j q^j  }{\qrfac{q}{j} \pqrfac{bq}{j}{q^2} }
= \frac{1}{\qrfac{aq}{\infty}\pqrfac{bq}{\infty}{q^2}} 
 \sum_{k= 0}^\infty \frac{\pqrfac{aq}{2k}{q} }{\pqrfac{q^2}{k}{q^2}}
(-b)^kq^{k^2} .
\end{equation}
When $h=1$ and $t=2$, \eqref{GB-1.4.3-4} reduces to \cite[Entry 1.4.4]{AB2009}:
 \begin{equation} \label{1.4.4}
\sum_{j=0}^\infty \frac{a^jq^{2j}  }{\pqrfac{q^2}{j}{q^2} \pqrfac{bq}{2j}{q} }
= \frac{1}{\pqrfac{aq^2}{\infty}{q^2}\pqrfac{bq}{\infty}{q}} 
 \sum_{k = 0}^\infty \frac{\pqrfac{aq^2}{k}{q^2} }{\pqrfac{q}{k}{q}}
(-b)^k q^{{k+1\choose 2}} .
\end{equation}

The case $a\to 0$ and $c=0$ case of \eqref{GB-1.4.1} is equivalent to \eqref{GB-1.4.3-4}, up to re-labeling of parameters. 

\subsection*{Entry 1.4.10 and Entry 1.4.11}
Both these entries immediately follow from $h=1=t$ case of \eqref{GB-1.4.3-4}. In this case,  when $a=1=b$ in \eqref{GB-1.4.3-4}, we obtain \cite[Entry 1.4.10]{AB2009}
\begin{equation}\label{1.4.10}
\sum_{j=  0}^\infty \frac{q^j}{\left(q;q\right)^2_j} =\frac{1}{\left(q;q\right)^2_{\infty}}
\sum_{k=0}^{\infty}
(-1)^{k} q^{k+1\choose 2} .
\end{equation}

Next, take $a=q^t$ and $b=q^{h-1}$ in \eqref{GB-1.4.3-4} to obtain
\begin{equation}\label{GB-1.4.11}
 \sum_{j=0}^{\infty} 
 \frac{ q^{2tj}}{\pqrfac{q^t}{j}{q^t} \pqrfac{q^h}{tj}{q^h}}
=
\frac{1}{\pqrfac{q^t}{\infty}{q^t}\pqrfac{q^h}{\infty}{q^h}}
 \sum_{k=0}^{\infty}  \frac{\pqrfac{q^t}{hk+1}{q^t} }{\pqrfac{q^h}{k}{q^h}}
 (-1)^k q^{h{k+1 \choose 2}} .
\end{equation}
Now take $h=1=t$ and simplify as follows.
\begin{align*}
\sum_{j=  0}^\infty \frac{q^{2j}}{\left(q;q\right)^2_j} & =\frac{1}{\left(q;q\right)^2_{\infty}}
\sum_{k=0}^{\infty}
(-1)^{k} q^{k+1\choose 2} \left(1-q^{k+1}\right) \cr
&=\frac{1}{\left(q;q\right)^2_{\infty}}
\left[
\sum_{k=0}^{\infty}
(-1)^{k} q^{k+1\choose 2} +
\sum_{k=0}^{\infty}
(-1)^{k+1} q^{k+2\choose 2} \right]
\cr
&=\frac{1}{\left(q;q\right)^2_{\infty}}
\left[ 1+
2\sum_{k=1}^{\infty}
(-1)^{k} q^{k+1\choose 2} \right].
\end{align*}
Ramanujan's Entry 1.4.11 is the first sum equated with the last in this chain of equalities 
\cite[Entry 1.4.11]{AB2009}:
\begin{align}\label{1.4.11}
\sum_{j=  0}^\infty \frac{q^{2j}}{\left(q;q\right)^2_j} =\frac{1}{\left(q;q\right)^2_{\infty}}
\left[ 1+
2\sum_{k=1}^{\infty}
(-1)^{k} q^{k+1\choose 2} \right].
\end{align}

\subsection*{Entry 1.4.12, Entry 1.4.17, Entry 1.4.9, and a part of Entry 1.5.1}
Consider the case $a\to 0$, $d\to 0$ of \eqref{GB-1.4.1}. In the resulting identity, bring the infinite product  $\pqrfac{-bq}{\infty}{q^t}$ to the other side, and cancel $(1-cq)$ on the LHS and $(1+bq)$ on the RHS. Then replace $c$ by $-a/q$ and $b$ by $b/q$ to obtain the appealing identity
\begin{align}\label{GB-1.4.12}
\pqrfac{-aq^h}{\infty}{q^h} &
\sum_{j=0}^{\infty} 
\frac {b^j q^{t{j+1\choose 2}} }{\pqrfac{q^t}{j}{q^t} \pqrfac{-aq^h}{tj}{q^h}} \cr
=&
\pqrfac{-bq^t}{\infty}{q^t} \sum_{k=0}^{\infty} 
\frac{ a^k q^{h{k+1\choose 2}} }{\pqrfac{q^h}{k}{q^h} \pqrfac{-bq^t}{hk}{q^t}} .
\end{align}
Many special cases of this symmetric identity have been found useful, some noted below, and one considered in \S\ref{sec:1.6.6}.

When $h=1$, \eqref{GB-1.4.12} reduces to \cite[Entry 1.4.12]{AB2009}
\begin{align}\label{1.4.12}
\pqrfac{-aq}{\infty}{q} &
\sum_{j=0}^{\infty} 
\frac {b^j q^{t{j+1\choose 2}} }{\pqrfac{q^t}{j}{q^t} \pqrfac{-aq}{tj}{q}} \cr
=&
\pqrfac{-bq^t}{\infty}{q^t} \sum_{k=0}^{\infty} 
\frac{ a^k q^{{k+1\choose 2}} }{\pqrfac{q}{k}{q} \pqrfac{-bq^t}{k}{q^t}} .
\end{align}
Take $h=t$ in \eqref{GB-1.4.12}, and then replace $q$ by $q^{1/t}$ to obtain \cite[Entry 1.4.17]{AB2009}
\begin{align}\label{1.4.17}
\pqrfac{-aq}{\infty}{q} &
\sum_{j=0}^{\infty} 
\frac {b^j q^{{j+1\choose 2}} }{\pqrfac{q}{j}{q} \pqrfac{-aq}{tj}{q}} \cr
=&
\pqrfac{-bq}{\infty}{q} \sum_{k=0}^{\infty} 
\frac{ a^k q^{{k+1\choose 2}} }{\pqrfac{q}{k}{q} \pqrfac{-bq}{tk}{q}} ,
\end{align}
one of Ramanujan's formulas highlighted in the introduction. 
Take $a=-1$ and $b=1$ in Ramanujan's Entry 1.4.17 (eq.\ \eqref{1.4.17}) to obtain
\begin{equation}\label{GB-1.4.9}
 \sum_{j=0}^\infty \frac{q^{{j+1\choose 2}}}{\pqrfac{q}{j}{q}\pqrfac{q}{tj}{q}} 
 =\frac{\pqrfac{-q}{\infty}{q}}{\left(q;q\right)_{\infty}}
\sum_{k=0}^{\infty}
 \frac{(-1)^{k} q^{{k+1\choose 2}}}{\pqrfac{q}{k}{q} \pqrfac{-q}{tk}{q}}.
\end{equation}
This further reduces to \cite[Entry 1.4.9]{AB2009} when $t=1$:
\begin{equation}\label{1.4.9}
 \sum_{j=0}^\infty \frac{q^{j+1\choose 2}}{\left(q;q\right)^2_j} 
 =\frac{\qrfac{-q}{\infty}}{\left(q;q\right)_{\infty}}
\sum_{k=0}^{\infty}
 \frac{(-1)^{k} q^{k+1\choose 2}}{\pqrfac{q^2}{k}{q^2}},
\end{equation}
where we use \cite[eq.\ (I.28)]{GR90}
\begin{equation*}
\pqrfac{a}{k}{q}\pqrfac{-a}{k}{q} =\pqrfac{a^2}{k}{q^2}
\end{equation*}
in the denominator of the RHS.

In Entry 1.4.17 (eq.\ \eqref{1.4.17}) replace $q$ by $q^2$ and take $t=1$ to obtain an assertion equivalent to an observation of M. Soros (see \cite[eq.\ (1.5.1)]{AB2009}):
\begin{align}\label{m.soros}
\pqrfac{-aq^2}{\infty}{q^2} &
\sum_{j=0}^{\infty} 
\frac {b^j q^{j^2+j} }{\pqrfac{q^2}{j}{q^2} \pqrfac{-aq^2}{j}{q^2}} \cr
=&
\pqrfac{-bq^2}{\infty}{q^2} \sum_{k=0}^{\infty} 
\frac{ a^k q^{k^2+k } }{\pqrfac{q^2}{k}{q^2} \pqrfac{-bq^2}{k}{q^2}}.
\end{align}
When we take $a\mapsto a/q$ and $b\mapsto b/q$ in \eqref{m.soros}, we obtain 
\cite[eq.~(1.5.1)]{AB2009}. Instead, if we take the special case $b\mapsto a/q$ in \eqref{m.soros}, we obtain the second equality of 
\cite[Entry 1.5.1]{AB2009}:
\begin{align}\label{1.5.1}
\pqrfac{-aq^2}{\infty}{q^2} &
\sum_{j=0}^{\infty} 
\frac {a^j q^{j^2} }{\pqrfac{q^2}{j}{q^2} \pqrfac{-aq^2}{j}{q^2}} \cr
=&
\pqrfac{-aq}{\infty}{q^2} \sum_{k=0}^{\infty} 
\frac{ a^k q^{k^2+k } }{\pqrfac{q^2}{k}{q^2} \pqrfac{-aq}{k}{q^2}}.
\end{align}

\subsection*{Entry 1.4.18} Entry 1.4.18 is due to Andrews, Berndt and Ramanujan \cite{AB2009}, and follows from the $a\to 0$ case of Ramanujan's transformation \eqref{1.4.1}. 

Take $a\to 0$ in 
\eqref{GB-1.4.1}, cancel the factor $(1-cq)$ from the LHS, and $(1+bq)$ from both sides. In the resulting identity take $b\mapsto b/q$, $c\mapsto a/q$, and $d\mapsto -c/q^h$, and bring the infinite products to the other side, we obtain: for $|c|<1$,
\begin{align}\label{GB-1.4.18}
\sum_{j=0}^{\infty} & \frac{\pqrfac{-c}{tj}{q^h}}  {\pqrfac{q^t}{j}{q^t} \pqrfac{aq^h}{tj}{q^h}}
 b^j q^{t{j+1\choose 2}} \cr
&=\frac{\pqrfac{-bq^t}{\infty}{q^t} \pqrfac{-c}{\infty}{q^h} }{\pqrfac{aq^h}{\infty}{q^h}} 
\sum_{k=0}^{\infty} \frac{\pqrfac{-aq^h/c}{k}{q^h}}{\pqrfac{q^h}{k}{q^h}\pqrfac{-bq^t}{hk}{q^t}}
(-c)^k  .
\end{align}
When $c=a/b$, $h=2$ and $t=1$, this reduces to \cite[Entry 1.4.18]{AB2009}: 
\begin{align}\label{1.4.18}
\sum_{j=0}^{\infty} & \frac{\pqrfac{-a/b}{j}{q^2}}  {\pqrfac{q}{j}{q} \pqrfac{aq^2}{j}{q^2}}
 b^j q^{{j+1\choose 2}} \cr
&=\frac{\pqrfac{-bq}{\infty}{q} \pqrfac{-a/b}{\infty}{q^2} }{\pqrfac{aq^2}{\infty}{q^2}} 
\sum_{k=0}^{\infty} \frac{\pqrfac{-bq^2}{k}{q^2}}{\pqrfac{q^2}{k}{q^2}\pqrfac{-bq}{2k}{q}}
\left(-\frac{a}{b}\right)^k .
\end{align}
Perhaps the $c=a/b$, $h=1$ and $t=2$ case of \eqref{GB-1.4.18} is equally pretty. 
\begin{align}\label{GB-1.4.18a}
\sum_{j=0}^{\infty} & \frac{\pqrfac{-a/b}{2j}{q}}  {\pqrfac{q^2}{j}{q^2} \pqrfac{aq}{2j}{q}}
 b^j q^{j^2+j} \cr
&=\frac{\pqrfac{-bq^2}{\infty}{q^2} \pqrfac{-a/b}{\infty}{q} }{\pqrfac{aq}{\infty}{q}} 
\sum_{k=0}^{\infty} \frac{\pqrfac{-bq}{k}{q}}{\pqrfac{q}{k}{q}\pqrfac{-bq^2}{k}{q^2}}
\left(-\frac{a}{b}\right)^k .
\end{align}

To get other transformations of a similar nature, consider the $b=0$ case of \eqref{GB-1.4.1}. 

Take $b = 0$ in 
\eqref{GB-1.4.1}, cancel the factor $(1-cq)$ from the LHS. In the resulting identity take $a\mapsto -b$,  $c\mapsto a/q$, and $d\mapsto -c/q^h$, and bring the infinite products to the other side.
We obtain: for  $|bq^t|<1$, $|c|<1$,
\begin{align}\label{GB-1.4.18-b0}
\sum_{j=0}^{\infty} & \frac{\pqrfac{-c}{tj}{q^h}}  {\pqrfac{q^t}{j}{q^t} \pqrfac{aq^h}{tj}{q^h}}
(-bq^t)^j  \cr
&=\frac{ \pqrfac{-c}{\infty}{q^h} }{\pqrfac{aq^h}{\infty}{q^h}\pqrfac{-bq^t}{\infty}{q^t}} 
\sum_{k=0}^{\infty} \frac{\pqrfac{-aq^h/c}{k}{q^h}\pqrfac{-bq^t}{hk}{q^t}}{\pqrfac{q^h}{k}{q^h}}
(-c)^k  .
\end{align}
Take the $c=a/b$, $h=2$, $t=1$ case of \eqref{GB-1.4.18-b0}  to obtain
\begin{align}\label{GB-1.4.18-b0a}
\sum_{j=0}^{\infty} & \frac{\pqrfac{-a/b}{j}{q^2}}  {\pqrfac{q}{j}{q} \pqrfac{aq^2}{j}{q^2}}
(-bq)^j  \cr
&=\frac{ \pqrfac{-a/b}{\infty}{q^2} }{\pqrfac{aq^2}{\infty}{q^2}\pqrfac{-bq}{\infty}{q}} 
\sum_{k=0}^{\infty} \frac{\pqrfac{-bq^2}{k}{q^2}\pqrfac{-bq}{2k}{q}}{\pqrfac{q^2}{k}{q^2}}
\left(-\frac{a}{b}\right)^k .
\end{align}

\subsection*{An equivalent form of Entry 1.6.5}
Consider the case $a\to 0$ of \eqref{GB-1.4.1}. Replace $c$ by $dq^{h-1}$, and then take $b\mapsto b/q$ and $d\mapsto -a$ to obtain, for $|aq^h|<1$,
\begin{align}\label{GB-1.6.5}
\sum_{j=0}^{\infty} 
\frac {b^j q^{t{j+1\choose 2}} }{\pqrfac{q^t}{j}{q^t} \left( 1+aq^{h(tj+1)} \right) }
=
\pqrfac{-bq^t}{\infty}{q^t} \sum_{k=0}^{\infty} 
\frac{ (-1)^k \left(aq^h\right)^k  }{\pqrfac{-bq^t}{hk}{q^t}} .
\end{align}
In the case that $h=1$, $t=2$ and $b=a$ this reduces to an equivalent form of Entry 1.6.5, the second last equation in the proof of  \cite[Entry 1.6.5]{AB2009}:
\begin{align}\label{1.6.5-eq}
\sum_{j=0}^{\infty} 
\frac {a^j q^{j^2+j} }{\pqrfac{q^2}{j}{q^2} \left( 1+aq^{2j+1} \right) }
=
\pqrfac{-aq^2}{\infty}{q^2} \sum_{k=0}^{\infty} 
\frac{ \left(-aq\right)^k  }{\pqrfac{-aq^2}{k}{q^2}} .
\end{align}
This identity has a combinatorial proof, given by Berndt, Kim and Yee \cite[Th. 5.7]{BKY2010}.
A very similar identity is obtained when $h=2$, $t=1$, $a\mapsto a/q$, and $b=a$ in \eqref{GB-1.6.5}:
\begin{align}\label{GB-1.6.5a}
\sum_{j=0}^{\infty} 
\frac {a^j q^{j+1\choose 2} }{\pqrfac{q}{j}{q} \left( 1+aq^{2j+1} \right) }
=
\pqrfac{-aq}{\infty}{q} \sum_{k=0}^{\infty} 
\frac{  \left(-aq\right)^k  }{\pqrfac{-aq}{2k}{q}} .
\end{align}

Similar identities are obtained when $b=0$ in \eqref{GB-1.4.1}. Set $d=c/q^{h-1}$ and in the resulting identity, relabel parameters by replacing $c$ by $-aq^{h-1}$ and $a$ by $-b$, to obtain, for $|aq^h|<1$, $|bq^t|<1$,
\begin{align}\label{GB-1.6.5e}
\sum_{j=0}^{\infty} 
\frac {\left(-bq^t\right)^j  }{\pqrfac{q^t}{j}{q^t} \left( 1+aq^{h(tj+1)} \right) }
=
\frac{1}{\pqrfac{-bq^t}{\infty}{q^t}} \sum_{k=0}^{\infty} 
\pqrfac{-bq^t}{hk}{q^t}
  \left(-aq^h\right)^k   .
\end{align}
Now take $b=a$, $h=1$ and $t=2$ to find that 
\begin{align}\label{GB-1.6.5f}
\sum_{j=0}^{\infty} 
\frac {(-aq^2)^j }{\pqrfac{q^2}{j}{q^2} \left( 1+aq^{2j+1} \right) }
=
\frac{1}{\pqrfac{-aq^2}{\infty}{q^2} }
\sum_{k=0}^{\infty} 
\pqrfac{-aq^2}{k}{q^2}
  \left(-aq\right)^k   .
\end{align}

\subsection*{The $b=0=c$ case of \eqref{GB-1.4.1}}
Take $b=0=c$ in \eqref{GB-1.4.1}, replace $d$ by $b$ and take $\pqrfac{aq^t}{\infty}{q^t}$ on the other side to obtain an extremely symmetric  transformation formula; for $|aq^t|<1$ and $|bq^h|<1$,
\begin{align}\label{GB-missing1}
\frac{1}{\pqrfac{bq^h}{\infty}{q^h}}
\sum_{j=0}^{\infty} 
\frac {\pqrfac{bq^h}{tj}{q^h} }{\pqrfac{q^t}{j}{q^t} }(aq^t)^{j} 
=
\frac{1}{\pqrfac{aq^t}{\infty}{q^t} }
\sum_{k=0}^{\infty} 
\frac{   \pqrfac{aq^t}{hk}{q^t}}{\pqrfac{q^h}{k}{q^h} } (bq^h)^k  .
\end{align}
There is no corresponding formula in Ramanujan's list appearing in Chapter 1 of \cite{AB2009}, but it is related to \cite[Cor.~1.2.2]{AB2009}, a result originally due to Andrews \cite{An1966b}. 
When $h=2$ and $t=1$, then \eqref{GB-missing1} reduces to
\begin{align}\label{GB-missing1a}
\frac{1}{\pqrfac{bq^2}{\infty}{q^2}}
\sum_{j=0}^{\infty} 
\frac {\pqrfac{bq^2}{j}{q^2} }{\pqrfac{q}{j}{q} }(aq)^{j} 
=
\frac{1}{\pqrfac{aq}{\infty}{q} }
\sum_{k=0}^{\infty} 
\frac{   \pqrfac{aq}{2k}{q}}{\pqrfac{q^2}{k}{q^2} } (bq^2)^k .
\end{align}
Compare the sum on the left hand side with that of \cite[Cor.~1.2.2]{AB2009}. To obtain Cor.~1.2.2, Andrews and Berndt apply Heine's transformation once again on the right hand side of \eqref{GB-missing1a}.

\subsection*{Summary of special cases}
So far, we have listed 13 entries that are immediate special or limiting cases of \eqref{GB-1.4.1}. One more will appear in \S\ref{sec:1.6.6}. The main special case is Entry 1.4.1 (eq.\ \eqref{1.4.1}) which is the $h=2$ and $t=1$ case of \eqref{GB-1.4.1}. The
others are: 
\begin{enumerate}
\item The case $a=1=d$. This leads to Entry 1.4.2 and Entry 1.4.5.
\item The case $b=0$ and $d\to 0$ of \eqref{GB-1.4.1}. This leads to Entry 1.4.3 and Entry 1.4.4. Note that the case $a\to 0$ and $c=0$ leads to the same identities. Other special cases include Entry 1.4.10 and Entry 1.4.11.  
\item Taking $a\to 0$ (without changing $c$) leads to Entry 1.4.18, and an equivalent form of Entry 1.6.5. See also Entry 1.6.6 in \S\ref{sec:1.6.6} below. We have also taken $b=0$ for the sake of completeness. (The $d\to 0$ and $c=0$ cases are equivalent due to the symmetry of \eqref{GB-1.4.1}.)
\item  The case $a\to 0$ and $d\to 0$. This leads to Entry 1.4.12, Entry 1.4.17, Entry 1.4.9, and a part of Entry 1.5.1.  
\item The case $b=0=c$. This leads to a new transformation formula. A special case is closely related to a useful transformation formula of Andrews in \cite[Cor.~1.2.2]{AB2009}. 
\end{enumerate}
By examining the above summary carefully, one can ask about the cases when $b=0$ followed by $a=0$,  or $c=0$ followed by $d=0$ in \eqref{GB-1.4.1}. However, in both these cases the resulting identity reduces to the $q$-binomial theorem. 

It is apparent that most of Ramanujan's identities considered here are simple limiting cases of \eqref{GB-1.4.1} where one or more parameters go to $0$.  However, there are a few that are motivated by getting a $q$-series (such as Entry 1.4.10 and Entry 1.4.11), or in getting an \lq unnatural' identity, where the factor that naturally terminates the series from below is missing. See Entry 1.4.2 (eq.\ \eqref{1.4.2}), Entry 1.4.5 (eq.\ \eqref{1.4.5}) and \eqref{1.6.5-eq}.  

\section{Entry 1.6.6 and related summations}\label{sec:1.6.6}
Entry 1.6.6 is also a special case of \eqref{GB-1.6.5},  our generalization of Entry 1.6.5 above, and so of \eqref{GB-1.4.1}. What is different here is that one can employ a special case of the $q$-binomial theorem to sum one of the series. The special case we need is \cite[eq. (II.1)]{GR90}: for $|z|<1$,
\begin{equation}\label{GR90-II.1}  
\frac{1}{\qrfac{z}{\infty}} = 
\sum_{k=0}^{\infty} 
\frac{z^k}{\qrfac{q}{k}}.
\end{equation}
Observe that when $b=-1$ and $h=1$ in \eqref{GB-1.6.5}, then using \eqref{GR90-II.1}, we obtain
\begin{align*}
\sum_{j=0}^{\infty} 
\frac {(-1)^j q^{t{j+1\choose 2}} }{\pqrfac{q^t}{j}{q^t} \left( 1+aq^{tj+1} \right) }
&=
\pqrfac{q^t}{\infty}{q^t} \sum_{k=0}^{\infty} 
\frac{ (-1)^k \left(aq\right)^k  }{\pqrfac{q^t}{k}{q^t}} \cr
&= \frac{\pqrfac{q^t}{\infty}{q^t}}{\pqrfac{-aq}{\infty}{q^t}}.
\end{align*}
Replace $a$ by $aq^{s-1}$ to re-write this identity in the form
\begin{align}\label{GB-1.6.6}
\sum_{j=0}^{\infty} 
\frac {(-1)^j q^{t{j+1\choose 2}} }{\pqrfac{q^t}{j}{q^t} \left( 1+aq^{tj+s} \right) }
= \frac{\pqrfac{q^t}{\infty}{q^t}}{\pqrfac{-aq^s}{\infty}{q^t}}.
\end{align}
In the case where $a=-1$, $s=1$ and $t=2$, this reduces to Ramanujan's \cite[Entry 1.6.6]{AB2009}, an identity we highlighted in the introduction:
\begin{align}\label{1.6.6}
\sum_{j=0}^{\infty} 
\frac {(-1)^j q^{j^2+j} }{\pqrfac{q^2}{j}{q^2} \left( 1-q^{2j+1} \right) }
= \frac{\pqrfac{q^2}{\infty}{q^2}}{\pqrfac{q}{\infty}{q^2}},
\end{align}
where the ratio of infinite products on the right hand side are equal to Ramanujan's theta function $\psi(q)$, defined as
$$\psi(q):= \sum_{k=0}^\infty q^{\frac{k(k+1)}{2}}.$$
However, if we take $a=1$, $t=1$ and $s=1$ in \eqref{GB-1.6.6}, we obtain
\begin{align}\label{GB-1.6.6a}
\sum_{j=0}^{\infty} 
\frac {(-1)^j q^{{j+1\choose 2}} }{\pqrfac{q}{j}{q} \left( 1+q^{j+1} \right) }
= \frac{\pqrfac{q}{\infty}{q}}{\pqrfac{-q}{\infty}{q}},
\end{align}
where now the products on the right hand side are (in Ramanujan's notation) $\phi(-q)$, defined by
$$\phi(-q):= \sum_{k=-\infty}^\infty (-1)^kq^{k^2}.$$
So in \eqref{GB-1.6.6} we have a common generalization of \eqref{1.6.6} and \eqref{GB-1.6.6a}.

More generally, when $b=-1$, $a\mapsto -a^h$, then \eqref{GB-1.6.5} reduces to
\begin{align*}
\sum_{j=0}^{\infty} 
\frac {(-1)^j q^{t{j+1\choose 2}} }{\pqrfac{q^t}{j}{q^t} \left( 1-a^hq^{h(tj+1)} \right) }
=
\pqrfac{q^t}{\infty}{q^t} 
\sum_{k=0}^{\infty} 
\frac{  \left(aq\right)^{hk}  }{\pqrfac{q^t}{hk}{q^t}} .
\end{align*}
When $h$ is a positive integer, the sum on the right consists of every $h$th term of the summand in \eqref{GR90-II.1}. There is a simple trick to compute such a sum. It uses the fact that
\begin{equation*}
\frac{1}{h}\sum_{r=0}^{h-1} \omega_h^{rk} =
\begin{cases} 
1 & \text{ if } h\vert k \\
0 & \text{ otherwise}
\end{cases},
\end{equation*}
for $\omega_h = e^{2\pi i/h}$ or some other primitive $h$th  root of unity. Using this trick, we find that
\begin{align*}
\sum_{k=0}^{\infty} 
\frac{  \left(aq\right)^{hk}  }{\pqrfac{q^t}{hk}{q^t}} &= 
\sum_{k=0}^{\infty} 
\frac{  \left(aq\right)^{k}  }{\pqrfac{q^t}{k}{q^t}} \frac{1}{h}\sum_{r=0}^{h-1} \omega_h^{rk} \cr
&= 
\frac{1}{h} \sum_{r=0}^{h-1}
\sum_{k=0}^{\infty} 
\frac{  \left(aq\omega_h^{r} \right)^{k}  }{\pqrfac{q^t}{k}{q^t}}  \cr
&= 
\frac{1}{h} \sum_{r=0}^{h-1}
\frac{ 1 }{\pqrfac{aq\omega_h^r}{\infty}{q^t}}.
\end{align*}
So we obtain
\begin{align}\label{GB-1.6.5b}
\sum_{j=0}^{\infty} 
\frac {(-1)^j q^{t{j+1\choose 2}} }{\pqrfac{q^t}{j}{q^t} \left( 1-a^hq^{h(tj+1)} \right) }
=
\frac{1}{h} \sum_{r=0}^{h-1}
\frac{ \pqrfac{q^t}{\infty}{q^t} }{\pqrfac{aq\omega_h^r}{\infty}{q^t}}.
\end{align}
In particular when $h=2$, then $\omega_h =-1$, and
\begin{align}\label{GB-1.6.5c}
\sum_{j=0}^{\infty} 
\frac {(-1)^j q^{t{j+1\choose 2}} }{\pqrfac{q^t}{j}{q^t} \left( 1-a^2q^{2(tj+1)} \right) }
=
\frac{1}{2} 
\left(
\frac{ \pqrfac{q^t}{\infty}{q^t} }{\pqrfac{aq}{\infty}{q^t}}
+
\frac{ \pqrfac{q^t}{\infty}{q^t} }{\pqrfac{-aq}{\infty}{q^t}}
\right).
\end{align}
When $a=1$ and $t=2$, this reduces to 
\begin{align}\label{GB-1.6.5d}
\sum_{j=0}^{\infty} 
\frac {(-1)^j q^{j^2+j } }{\pqrfac{q^2}{j}{q^2} \left( 1-q^{2(2j+1)} \right) }
&=
\frac{1}{2} 
\left(
\frac{ \pqrfac{q^2}{\infty}{q^2} }{\pqrfac{q}{\infty}{q^2}}
+
\frac{ \pqrfac{q^2}{\infty}{q^2} }{\pqrfac{-q}{\infty}{q^2}}\right) \cr
&=\frac{1}{2}\left( \psi(q) + \psi(-q)\right)
.
\end{align}
The right hand side is the even part of $\psi (q)$.

\begin{Remark} Let $\omega_r = e^{2\pi i/r}$ or some other primitive $r$th  root of unity. Then we   can show that, formally,
\begin{equation*}
\frac{1}{r}\sum_{\nu =0}^{r-1} \omega_r^{-\nu s} \sum_{n=0}^\infty f(n)\omega_r^{\nu n} =
\sum_{k=0}^\infty f(rk+s)\ \frac{1}{r}\sum_{\nu=0}^{r-1} \omega_r^{\nu rk} 
=\sum_{k=0}^\infty f(rk+s).
\end{equation*}
This \lq sectioning'  process allows us to compute the sum $\sum_{k=0}^\infty f(rk+s)z^{rk+s},$ if we know the sum $\sum_{k=0}^\infty f(k)z^{k}$. 

This trick is used in the proof of Andrews' Fundamental Lemma, given in equation \eqref{Andrews-FL}. 
\end{Remark}
We have seen an example where we can sum a series after applying Heine's method.  Next, we obtain a result for multiple series by iterating Heine's method.

\section{Multibasic Andrews $q$-Lauricella transformation}\label{sec:lauricella}
We now apply Heine's method to obtain a multibasic generalization of Andrews'~\cite[(eq.~(4.1)]{An1972} transformation formula for the $q$-Lauricella function. As special cases, we obtain some generalizations of Entry 1.4.10 and of equation \eqref{GB-1.4.12}.  These results transform a multiple series to a multiple of a single series. 

For $\boldsymbol h$ and $\boldsymbol k$ vectors, we use the following notations. The notation 
$\sumk$ is used to denote the sum of the components of the vector
$k_1+k_2+\cdots +k_m.$
We use the symbol for the dot product
$$ {{\boldsymbol h}\cdot{\boldsymbol k}} = h_1k_1+h_2k_2+\cdots + h_mk_m.$$
We also use the vector $\boldsymbol {\delta}$ to denote the vector
$ (1, 2, \dots, m).$
Thus, 
$${\boldsymbol \delta}\cdot{\boldsymbol k} =\sum\limits_{r=1}^m rk_r .$$

\begin{Theorem} Suppose $m=1, 2, \dots$ is a non-negative integer.  Let $a_1, a_2, \dots , a_m$  and $b$ be complex numbers, and suppose that the denominators in \eqref{GB-qlauricella} are not zero. Further, 
let $q$, $t$, $h_1, h_2, \dots h_m$ be complex numbers, satisfying 
$|q^t|<1$, $|q^{h_r}|<1$ and $|q^{th_r}|<1$, for $r=1, 2, \dots, m$. 
 Then, for $|w|<1$, and $|z_r|<1$ (for $r=1, 2, \dots, m)$,
\begin{align}\label{GB-qlauricella}
\multisum{k}{m} \ &
\smallprod{m} \frac{\pqrfac{a_r}{k_r}{q^{h_r}}}{\pqrfac{q^{h_r}}{k_r}{q^{h_r}}}
\frac{\pqrfac{w}{ {\boldsymbol h}\cdot{\boldsymbol k}}{q^t}}
{\pqrfac{bw}{{\boldsymbol h}\cdot{\boldsymbol k}}{q^t}}
 \smallprod{m} z_r^{k_r} \cr
& =
\frac{\pqrfac{w}{\infty}{q^t}}{\pqrfac{bw}{\infty}{q^t}}
\smallprod{m} \frac{\pqrfac{a_rz_r}{\infty}{q^{h_r}}}{\pqrfac{z_r}{\infty}{q^{h_r}}}
\sum_{j=0}^{\infty} \frac{\pqrfac{b}{j}{q^t}} {\pqrfac{q^t}{j} {q^t}}
\smallprod{m} \frac{\pqrfac{z_r}{tj}{q^{h_r}}}{\pqrfac{a_rz_r}{tj}{q^{h_r}}}
w^j.
\end{align}
\end{Theorem}
\begin{Remark}
When $m=1$, then \eqref{GB-qlauricella} reduces to \eqref{GB-sym-Heine}. When $a_r\mapsto b_r$, $b\mapsto c/a$, $w\mapsto a$, and  $h_1=h_2=\cdots h_m=1=t$, then \eqref{GB-qlauricella} reduces to a transformation of Andrews for $q$-Lauricella functions 
\cite[(eq.~(4.1)]{An1972}. 
\end{Remark}
\begin{proof}
The proof of \eqref{GB-qlauricella} is a direct extension of the proof of \eqref{GB-sym-Heine}.  
For $m=1$, it reduces to Theorem \eqref{th:GB-sym-Heine}. 
When $m>1$, we apply Heine's method $m$ times. Expand the relevant products using the $q$-binomial theorem and interchange the sums one at a time. The first few steps of the proof are as follows.  
\begin{align*}
\multisum{k}{m} \ &
\smallprod{m} \frac{\pqrfac{a_r}{k_r}{q^{h_r}}}{\pqrfac{q^{h_r}}{k_r}{q^{h_r}}}
\frac{\pqrfac{w}{ {\boldsymbol h}\cdot{\boldsymbol k}}{q^t}}
{\pqrfac{bw}{{\boldsymbol h}\cdot{\boldsymbol k}}{q^t}}
 \smallprod{m} z_r^{k_r} \cr
=
\frac{\pqrfac{w}{\infty}{q^t}}{\pqrfac{bw}{\infty}{q^t}}
\multisum{k}{m} \ &
\smallprod{m} \frac{\pqrfac{a_r}{k_r}{q^{h_r}}}{\pqrfac{q^{h_r}}{k_r}{q^{h_r}}}
\frac{\pqrfac{bwq^{t\left(h_1k_1+\cdots+h_mk_m \right)}}{ \infty}{q^t}}
{\pqrfac{wq^{t\left(h_1k_1+\cdots+h_mk_m \right)}}{\infty}{q^t}}
 \smallprod{m} z_r^{k_r} \cr
\cr
 =
\frac{\pqrfac{w}{\infty}{q^t}}{\pqrfac{bw}{\infty}{q^t}}
\multisum{k}{m} \ &
\smallprod{m} \frac{\pqrfac{a_r}{k_r}{q^{h_r}}}{\pqrfac{q^{h_r}}{k_r}{q^{h_r}}}
\smallprod{m} z_r^{k_r} 
\sum_{j=0}^{\infty} \frac{\pqrfac{b}{j}{q^t}} {\pqrfac{q^t}{j} {q^t}}
w^{j}q^{tj\left(h_1k_1+\cdots+h_mk_m \right)}
  \cr
 =
\frac{\pqrfac{w}{\infty}{q^t}}{\pqrfac{bw}{\infty}{q^t}}
\multisum{k}{m-1} \ &
\smallprod{m-1} \frac{\pqrfac{a_r}{k_r}{q^{h_r}}}{\pqrfac{q^{h_r}}{k_r}{q^{h_r}}}
\smallprod{m-1} z_r^{k_r} \cr
\times \sum_{j=0}^{\infty} \frac{\pqrfac{b}{j}{q^t}} {\pqrfac{q^t}{j} {q^t}} &
w^{j}q^{tj\left(h_1k_1+\cdots+h_{m-1}k_{m-1} \right)} 
 \sum_{k_m\geq 0} 
\frac{\pqrfac{a_m}{k_m}{q^{h_m}}}{\pqrfac{q^{h_m}}{k_m}{q^{h_m}}}
\left(z_mq^{tjh_m}\right)^{k_m}
  \cr
=
\frac{\pqrfac{w}{\infty}{q^t}}{\pqrfac{bw}{\infty}{q^t}}
 \multisum{k}{m-1} &
\smallprod{m-1} \frac{\pqrfac{a_r}{k_r}{q^{h_r}}}{\pqrfac{q^{h_r}}{k_r}{q^{h_r}}}
\smallprod{m-1} z_r^{k_r} \cr
\times \sum_{j=0}^{\infty} & \frac{\pqrfac{b}{j}{q^t}} {\pqrfac{q^t}{j} {q^t}}
w^{j}q^{tj\left(h_1k_1+\cdots+h_{m-1}k_{m-1} \right)} 
\frac{\pqrfac{a_mz_mq^{tjh_m}}{\infty}{q^{h_m}}}{\pqrfac{z_mq^{tjh_m}}{\infty}{q^{h_m}}}
  \cr
=
\frac{\pqrfac{w}{\infty}{q^t}}{\pqrfac{bw}{\infty}{q^t}} 
\frac{\pqrfac{a_mz_m}{\infty}{q^{h_m}}}{\pqrfac{z_m}{\infty}{q^{h_m}}} &
\multisum{k}{m-1} \ 
\smallprod{m-1} \frac{\pqrfac{a_r}{k_r}{q^{h_r}}}{\pqrfac{q^{h_r}}{k_r}{q^{h_r}}}
\smallprod{m-1} z_r^{k_r} \cr
&\times 
\sum_{j=0}^{\infty} \frac{\pqrfac{b}{j}{q^t}} {\pqrfac{q^t}{j} {q^t}}
\frac{\pqrfac{z_m}{tj}{q^{h_m}}}{\pqrfac{a_mz_m}{tj}{q^{h_m}}}
w^{j}q^{tj\left(h_1k_1+\cdots+h_{m-1}k_{m-1} \right)} .
\end{align*}
So far, the steps are identical to the proof of Theorem \ref{th:GB-sym-Heine}, with the sum indexed by $k_m$ replacing the sum indexed by $k$ in the earlier proof. Repeating these steps $m-1$ times, we obtain the required single sum indexed by $j$ on the right hand side of \eqref{GB-qlauricella}. 

The convergence considerations in \S\ref{sec:heine-method} extend to both the series in this theorem, and to the interchange of summations required in the proof. 
\qed
\end{proof}

Next, we indicate generalizations of a few special cases of results from \S \ref{sec:ramanujan2p1}, to hint at the many possibilities available.  

First take $a_r\mapsto c_rq/d_r$, $b\mapsto -bq/a$, $w\mapsto aq^t$ and $z_r\mapsto d_rq^{h_r}$ in \eqref{GB-qlauricella} to obtain a generalization of \eqref{GB-1.4.1}:
\begin{align}\label{GB-qlauricella-1.4.1}
\frac{\pqrfac{aq^t}{\infty}{q^t}}{\pqrfac{-bq^{t+1}}{\infty}{q^t}}&
\smallprod{m} \frac{\pqrfac{c_rq^{h_r+1}}{\infty}{q^{h_r}}}{\pqrfac{d_rq^{h_r}}{\infty}{q^{h_r}}}
\sum_{j=0}^{\infty} \frac{\pqrfac{-bq/a}{j}{q^t}} {\pqrfac{q^t}{j} {q^t}}
\smallprod{m} \frac{\pqrfac{d_rq^{h_r}}{tj}{q^{h_r}}}{\pqrfac{c_rq^{h_r+1}}{tj}{q^{h_r}}}
\left(aq^t\right)^j\cr
 =
&\multisum{k}{m} 
\smallprod{m} \frac{\pqrfac{c_rq/d_r}{k_r}{q^{h_r}}}{\pqrfac{q^{h_r}}{k_r}{q^{h_r}}}
\frac{\pqrfac{aq^t}{{\boldsymbol h}\cdot{\boldsymbol k}}{q^t}}
{\pqrfac{-bq^{t+1}}{{\boldsymbol h}\cdot{\boldsymbol k}}{q^t}}
q^{{\boldsymbol h}\cdot{\boldsymbol k}}
 \smallprod{m} d_r^{k_r} .
\end{align}

Next, we obtain four generalizations of Ramanujan's Entry 1.4.10, equation \eqref{1.4.10}. 

First take $b=0$ and $d_r\to 0$, for $r=1, 2, \dots, m$ in \eqref{GB-qlauricella-1.4.1}. Further, replace $c_r$ by $1/q$ for each $r$, and take $a=1$. In the resulting identity, take $t=1$, and
$h_r=r$ for $r=1, 2, \dots, m$ to obtain
\begin{align}\label{GB-qlauricella-1.4.10a}
\sum_{j=0}^{\infty} & \frac{q^j} {\pqrfac{q}{j} {q}}
\smallprod{m} \frac{1}{\pqrfac{q^{r}}{j}{q^{r}}}
  =
\frac{1}{\pqrfac{q}{\infty}{q}} 
 \smallprod{m} \frac{1}{\pqrfac{q^r}{\infty}{q^{r}}}\cr
& \times\multisum{k}{m} 
\smallprod{m} \frac{ 1 }{\pqrfac{q^r}{k_r}{q^{r}}}
\pqrfac{q}{{\boldsymbol \delta}\cdot{\boldsymbol k}}{q}  (-1)^{|{\boldsymbol k}|}
q^{\sum\limits_{r=1}^m r{k_r+1\choose 2}}.
\end{align}

Next, again take $b=0$ and $d_r\to 0$, for $r=1, 2, \dots, m$ in \eqref{GB-qlauricella-1.4.1}. But now take $m=n$, and replace $h_r$ by $n$ for all $r$. Further, set $c_r = cq^{r-2}$ for $r=1, 2, \dots, n$, and invoke \eqref{GRI.27} to simplify some of the products. Finally, take $a=c=1$ to obtain the following generalization of Entry 1.4.10:
\begin{align}\label{GB-qlauricella-1.4.10b}
\sum_{j=0}^{\infty}& \frac{q^{tj}} {\pqrfac{q^t}{j} {q^t} \pqrfac{q^{n}}{ntj}{q}}
 =
\frac{1}{\pqrfac{q^t}{\infty}{q^t} \pqrfac{q^n}{\infty}{q} } 
\cr
& \times\multisum{k}{n} 
\smallprod{n} \frac{ 1 }{\pqrfac{q^n}{k_r}{q^{n}}}
\pqrfac{q^t}{n|{\boldsymbol k}|}{q^t} (-1)^{|{\boldsymbol k}|}
q^{\sum\limits_{r=1}^n (r-1)k_r + n\sum\limits_{r=1}^n {k_r+1\choose 2}}.
\end{align}

The third generalization of Entry 1.4.10 is obtained as follows. In equation \eqref{GB-qlauricella-1.4.1}, take $a\to 0$ and $c_r=0$ for $r=1, 2, \dots, m$. Now take $b\mapsto -1/q$, $d_r=1,$ for $r=1, 2, \dots, m$. In the resulting transformation, once again take $t=1$, and $h_r=r$, for $r=1, 2, \dots, m$, and obtain
\begin{align}\label{GB-qlauricella-1.4.10c}
\multisum{k}{m} &
\frac{1}{\pqrfac{q}{{\boldsymbol \delta}\cdot{\boldsymbol k}}{q}}
\smallprod{m} \frac{ 1 }{\pqrfac{q^r}{k_r}{q^{r}}}
q^{{\boldsymbol \delta}\cdot{\boldsymbol k}} \cr
&= 
\frac{1}{\pqrfac{q}{\infty}{q}} 
 \smallprod{m} \frac{1}{\pqrfac{q^r}{\infty}{q^{r}}}
\sum_{j=0}^{\infty} 
(-1)^{j}
q^{ j+1\choose 2}
\prod\limits_{r=2}^{m} \pqrfac{q^{r}}{j}{q^{r}}.
\end{align}

The fourth generalization of Entry 1.4.10 is as follows. Again, take $a\to 0$ and $c_r=0$ for $r=1, 2, \dots, m$. But now take $m=n$, and replace $h_r$ by $n$ for all $r$. Further, replace $b$ by $-b/q$, and  set $d_r = dq^{r-1}$ for $r=1, 2, \dots, n$. Again, we invoke \eqref{GRI.27} to simplify some of the products, and take $b = d = 1$ to obtain the following generalization of Entry 1.4.10.
\begin{align}\label{GB-qlauricella-1.4.10d}
\multisum{k}{n} &
\frac{1}{\pqrfac{q^t}{n{| {\boldsymbol k} |}}{q^t}}
\smallprod{n} \frac{ 1 }{\pqrfac{q^n}{k_r}{q^{n}}}
q^{n{| {\bold k} |}+\sum\limits_{r=1}^n (r-1)k_r} \cr
&= 
\frac{1}{\pqrfac{q^t}{\infty}{q^t} \pqrfac{q^n}{\infty}{q} } 
 \sum_{j=0}^{\infty} 
\frac{ \pqrfac{q^{n}}{ntj}{q} } {\pqrfac{q^{t}}{j}{q^{t}}}
(-1)^{j}
q^{t{ j+1\choose 2}}.
\end{align}

Next, we give a generalization of the generalization of Entry 1.4.12 given in equation \eqref{GB-1.4.12}. To this end take the limit as $a\to 0$ and $d_r\to 0$, for $r=1, 2, \dots, m$ in \eqref{GB-qlauricella-1.4.1}. Now replace $c_r$ by $-a_r/q$ and $b$ by $b/q$ to obtain
\begin{align}\label{GB-qlauricella-1.4.12a}
\smallprod{m} \pqrfac{-a_rq^{h_r}}{\infty}{q^{h_r}} 
&\sum_{j=0}^{\infty} \frac{1} {\pqrfac{q^t}{j} {q^t}}
\smallprod{m} \frac{1}{\pqrfac{-a_rq^{h_r}}{tj}{q^{h_r}}}
b^jq^{t{j+1\choose 2}}\cr
 =
\pqrfac{-bq^t}{\infty}{q^t}
&\multisum{k}{m} 
\smallprod{m} \frac{a_r^{k_r} }{\pqrfac{q^{h_r}}{k_r}{q^{h_r}}}
\frac{1}
{\pqrfac{-bq^t}{{\boldsymbol h}\cdot{\boldsymbol k}}{q^t}}
q^{\sum\limits_{r=1}^m h_r{k_r+1\choose 2}}.
\end{align}
Finally, we present a special case of \eqref{GB-qlauricella-1.4.12a} with $m=n$ and where $h_r=n$, for all $r$. Take $a_r=aq^{r-1}$ for all $r$ and simplify some of the products using \eqref{GRI.27} to obtain the following generalization of Entry 1.4.12:
\begin{align}\label{GB-qlauricella-1.4.12b}
& \pqrfac{-aq^n}{\infty}{q} 
\sum_{j=0}^{\infty} \frac{b^j} {\pqrfac{q^t}{j} {q^t} \pqrfac{-aq^n}{ntj}{q} }
q^{t{j+1\choose 2}}
 =
\pqrfac{-bq^t}{\infty}{q^t} \cr
&\times \multisum{k}{n} 
\smallprod{n} \frac{1 }{\pqrfac{q^n}{k_r}{q^n}}
\frac{a^{|{\boldsymbol k}|}}
{\pqrfac{-bq^t}{ n|{\boldsymbol k}|}{q^t}}
q^{\sum\limits_{r=1}^n (r-1)k_r+ n\sum\limits_{r=1}^n {k_r+1\choose 2}}.
\end{align}

Perhaps this is a suitable place to close our study, at equation number 60 of this paper.

\section{Closing remarks}\label{sec:conclusion}
We have seen that a minor modification of Andrews' earlier identity led to so many identities similar to Ramanujan's entries. 
Clearly, it is a good idea to study Ramanujan's Notebooks, edited by Berndt, and the Lost Notebook, edited by Andrews and Berndt.  We conclude with a few remarks regarding Ramanujan's transformations and possible directions of further study.  

Entry 1.4.1 is a key identity of Ramanujan, and deserves more importance than given in \cite{AB2009}. Recall that Entry 1.4.1 is the $h=2$ and $t=1$ case of \eqref{GB-1.4.1}. 
Many of Ramanujan's transformations considered here are immediate corollaries of Entry 1.4.1. These include Entries 1.4.2, 1.4.3, 1.4.4. 1.4.5 and 1.4.18. The special cases considered are the obvious ones, by letting one or more parameters equal to $0$, or if necessary, taking limits to $0$. Even the equivalent case of Entry 1.6.5 can be derived from Entry 1.4.1, by taking $d\to0$. 

Entries 1.4.9, 1.4.10, 1.4.11 and Entry 1.5.1 follow from the $h=1=t$ case of \eqref{GB-Heine} or in other words, from Heine's transformation \eqref{Heine-original}. Since Heine's transformation formula appears in earlier notebooks of Ramanujan, why do these formulas show up here, in his later work? An explanation is that Ramanujan was searching for identities for series that look like or involve theta functions. So these entries, and Entry 1.6.6 fit in well here. 

Next note that Entry 1.4.12 is obtained from the $h=1$ case of \eqref{GB-Heine},  while Entry 1.4.17 requires the the $h=t$ case of \eqref{GB-Heine}. Of the entries studied here, these two are the only ones that require something more than \eqref{Heine-original} and \eqref{1.4.1} (the two identities noted by Ramanujan in his notebooks).

Many of Ramanujan's identities have been studied from a partition theoretic perspective by Berndt, Kim and Yee \cite{BKY2010}, including  \eqref{1.4.12}, \eqref{1.4.17}, \eqref{1.4.9}, \eqref{1.5.1} and \eqref{1.6.5-eq}.   We expect that many of the identities considered here have a similar interpretation.

Finally, we note that
Heine's method generalizes to multiple series related to root systems. 
We can combine the multidimensional $q$-binomial theorems (given by, for example, 
Milne \cite{Milne1997} and Gustafson and Krattenthaler \cite{GK1996}) to obtain extensions of \eqref{GB-Heine} and \eqref{GB-qlauricella}.   We hope to present these elsewhere. 


\subsubsection*{Acknowledgements}{We thank Professor George Andrews and Professor Christian Krattenthaler for many suggestions,  pointers to useful references, and helpful discussions. }


\end{document}